\documentclass[12pt, reqno]{amsart}
\usepackage{amsmath, amsthm, amscd, amsfonts, amssymb, graphicx, color}
\usepackage[bookmarksnumbered, colorlinks, plainpages]{hyperref}
\hypersetup{colorlinks=true,linkcolor=red, anchorcolor=green, citecolor=cyan, urlcolor=red, filecolor=magenta, pdftoolbar=true}

\newtheorem{theorem}{Theorem}[section]
\newtheorem{lemma}[theorem]{Lemma}
\newtheorem{proposition}[theorem]{Proposition}
\newtheorem{corollary}[theorem]{Corollary}
\theoremstyle{definition}
\newtheorem{definition}[theorem]{Definition}
\newtheorem{example}[theorem]{Example}

\theoremstyle{remark}

\numberwithin{equation}{section}

\begin{document}

\setcounter{page}{1}

\title[Controlled ${\ast}$-operator Frames for $End_{\mathcal{A}}^{\ast}(\mathcal{H})$]{Controlled ${\ast}$ - Operator Frames for $End_{\mathcal{A}}^{\ast}(\mathcal{H})$}

\author[Abdeslam TOURI, Hatim LABRIGUI and , Samir KABBAJ]{Abdeslam TOURI$^1$$^{*}$, Hatim LABRIGUI$^1$ \MakeLowercase{and} Samir KABBAJ$^1$}

\address{$^{1}$Department of Mathematics, University of Ibn Tofail, B.P. 133, Kenitra, Morocco}
\email{\textcolor[rgb]{0.00,0.00,0.84}{touri.abdo68@gmail.com}}

\email{\textcolor[rgb]{0.00,0.00,0.84}{hlabrigui75@gmail.com}}

\email{\textcolor[rgb]{0.00,0.00,0.84}{samkabbaj@yahoo.fr}}



\date{
\newline \indent $^{*}$Corresponding author}

\begin{abstract}
	In this paper we
	study the concept of controlled $\ast$-operator frmae for $End_{\mathcal{A}}^{\ast}(\mathcal{H})$. Also we discuss characterizations of controlled $\ast$-operator frames and we give some properties.
\end{abstract} \maketitle
\section{Introduction and Preliminaries}
The concept of frames in Hilbert spaces has been introduced by
Duffin and Schaeffer \cite{Duf} in 1952 to study some deep problems in nonharmonic Fourier series. After the fundamental paper \cite{13} by Daubechies, Grossman and Meyer, frame theory began to be widely used, particularly in the more specialized context of wavelet frames and Gabor frames \cite{Gab}. Frames have been used in signal processing, image processing, data compression and sampling theory. 
 
Controlled frames in Hilbert spaces have been introduced by P. Balazs \cite{01} to improve the numerical efficiency of iterative algorithms for inverting the frame operator.

Controlled frames in $C^{\ast}$-modules were introduced by Rashidi and Rahimi \cite{03}, and the authors showed that they share many useful properties with their corresponding notions in a Hilbert space.  ${\ast}$-operator frmae for $End_{\mathcal{A}}^{\ast}(\mathcal{H})$ has been study by M. Rossafi \cite{haja}.
In this paper we introduce the notion of  controlled ${\ast}$-operator frame for $End_{\mathcal{A}}^{\ast}(\mathcal{H})$ with $\mathcal{H}$ is a Hilbert $C^{\ast}$-modules.

Let $I$ be a countable index set. In this section we briefly recall the definitions and basic properties of $C^{\ast}$-algebra, Hilbert $C^{\ast}$-modules, frame, operator frame in Hilbert $C^{\ast}$-modules. For information about frames in Hilbert spaces we refer to \cite{Ch}. Our references for $C^{\ast}$-algebras are \cite{Dav,Con}. For a $C^{\ast}$-algebra $\mathcal{A}$, an element $a\in\mathcal{A}$ is positive ($a\geq 0$) if $a=a^{\ast}$ and $sp(a)\subset\mathbf{R^{+}}$. $\mathcal{A}^{+}$ denotes the set of positive elements of $\mathcal{A}$.
\begin{definition}
	\cite{BA}. Let $ \mathcal{A} $ be a unital $C^{\ast}$-algebra and $\mathcal{H}$ be a left $ \mathcal{A} $-module, such that the linear structures of $\mathcal{A}$ and $ \mathcal{H} $ are compatible. $\mathcal{H}$ is a pre-Hilbert $\mathcal{A}$-module if $\mathcal{H}$ is equipped with an $\mathcal{A}$-valued inner product $\langle.,.\rangle_{\mathcal{A}} :\mathcal{H}\times\mathcal{H}\rightarrow\mathcal{A}$, such that is sesquilinear, positive definite and respects the module action. In the other words,
	\begin{itemize}
		\item [(i)] $ \langle x,x\rangle_{\mathcal{A}}\geq0 $ for all $ x\in\mathcal{H} $ and $ \langle x,x\rangle_{\mathcal{A}}=0$ if and only if $x=0$.
		\item [(ii)] $\langle ax+y,z\rangle_{\mathcal{A}}=a\langle x,y\rangle_{\mathcal{A}}+\langle y,z\rangle_{\mathcal{A}}$ for all $a\in\mathcal{A}$ and $x,y,z\in\mathcal{H}$.
		\item[(iii)] $ \langle x,y\rangle_{\mathcal{A}}=\langle y,x\rangle_{\mathcal{A}}^{\ast} $ for all $x,y\in\mathcal{H}$.
	\end{itemize}	 
	For $x\in\mathcal{H}, $ we define $||x||=||\langle x,x\rangle_{\mathcal{A}}||^{\frac{1}{2}}$. If $\mathcal{H}$ is complete with $||.||$, it is called a Hilbert $\mathcal{A}$-module or a Hilbert $C^{\ast}$-module over $\mathcal{A}$. For every $a$ in $C^{\ast}$-algebra $\mathcal{A}$, we have $|a|=(a^{\ast}a)^{\frac{1}{2}}$ and the $\mathcal{A}$-valued norm on $\mathcal{H}$ is defined by $|x|=\langle x, x\rangle_{\mathcal{A}}^{\frac{1}{2}}$ for $x\in\mathcal{H}$.
	\begin{example} \cite{Tro}
		If $ \{\mathcal{H}_{k}\}_{k\in\mathbf{N}} $ is a countable set of Hilbert $\mathcal{A}$-modules, then one can define their direct sum $ \oplus_{k\in\mathbb{N}}\mathcal{H}_{k} $. On the $\mathcal{A}$-module $ \oplus_{k\in\mathbb{N}}\mathcal{H}_{k} $ of all sequences $x=(x_{k})_{k\in\mathbb{N}}: x_{k}\in\mathcal{H}_{k}$, such that the series $ \sum_{k\in\mathbb{N}}\langle x_{k}, x_{k}\rangle_{\mathcal{A}} $ is norm-convergent in the $\mathcal{C}^{\ast}$-algebra $\mathcal{A}$, we define the inner product by
		\begin{equation*}
		\langle x, y\rangle:=\sum_{k\in\mathbb{N}}\langle x_{k}, y_{k}\rangle_{\mathcal{A}} 
		\end{equation*}
		for $x, y\in\oplus_{k\in\mathbb{N}}\mathcal{H}_{k} $.
		
		Hence $\oplus_{k\in\mathbb{N}}\mathcal{H}_{k}$ is a Hilbert $\mathcal{A}$-module.
		
		The direct sum of a countable number of copies of a Hilbert $\mathcal{C}^{\ast}$-module $\mathcal{H}$ is denoted by $l^{2}(\mathcal{H})$.
	\end{example}
\end{definition}
	Let $\mathcal{H}$ and $\mathcal{K}$ be two Hilbert $\mathcal{A}$-modules. A map $T:\mathcal{H}\rightarrow\mathcal{K}$ is said to be adjointable if there exists a map $T^{\ast}:\mathcal{K}\rightarrow\mathcal{H}$ such that $\langle Tx,y\rangle_{\mathcal{A}}=\langle x,T^{\ast}y\rangle_{\mathcal{A}}$ for all $x\in\mathcal{H}$ and $y\in\mathcal{K}$.
	 
	We also reserve the notation $End_{\mathcal{A}}^{\ast}(\mathcal{H},\mathcal{K})$ for the set of all adjointable operators from $\mathcal{H}$ to $\mathcal{K}$ and $End_{\mathcal{A}}^{\ast}(\mathcal{H},\mathcal{H})$ is abbreviated to $End_{\mathcal{A}}^{\ast}(\mathcal{H})$.

The following lemmas will be used to prove our mains results
\begin{lemma}\cite{Ali}\label{haja1}
	If $\varphi:\mathcal{A}\longrightarrow\mathcal{B}$ is a $\ast$-homomorphism between $\mathcal{C}^{\ast}$-algebras, then $\varphi$ is increasing, that is, if $a\leq b$, then $\varphi(a)\leq\varphi(b)$.
\end{lemma}
\begin{lemma}\label{haja2}
	\cite{Ali}. Let $\mathcal{H}$ and $\mathcal{K}$ be two Hilbert $\mathcal{A}$-modules and $T\in End_{\mathcal{A}}^{\ast}(\mathcal{H},\mathcal{K})$.
	\begin{itemize}
		\item [(i)] If $T$ is injective and $T$ has closed range, then the adjointable map $T^{\ast}T$ is invertible and $$\|(T^{\ast}T)^{-1}\|^{-1}I_{\mathcal{H}}\leq T^{\ast}T\leq\|T\|^{2}I_{\mathcal{H}}.$$
		\item  [(ii)]	If $T$ is surjective, then the adjointable map $TT^{\ast}$ is invertible and $$\|(TT^{\ast})^{-1}\|^{-1}I_{\mathcal{K}}\leq TT^{\ast}\leq\|T\|^{2}I_{\mathcal{K}}.$$
\end{itemize}	
\end{lemma}

\begin{lemma} \label{haja3} \cite{Pas}.
	Let $\mathcal{H}$ be Hilbert $\mathcal{A}$-module. If $T\in End_{\mathcal{A}}^{\ast}(\mathcal{H})$, then $$\langle Tx,Tx\rangle _{\mathcal{A}} \leq\|T\|^{2}\langle x,x\rangle _{\mathcal{A}}  , x\in\mathcal{H}.$$
\end{lemma}

\begin{lemma} \label{haja4} \cite{Ara}.
	Let $\mathcal{H}$ and $\mathcal{K}$ two Hilbert $\mathcal{A}$-modules and $T\in End^{\ast}(\mathcal{H},\mathcal{K})$. Then the following statements are equivalent:
	\begin{itemize}
		\item [(i)] $T$ is surjective.
		\item [(ii)] $T^{\ast}$ is bounded below with respect to norm, i.e., there is $m>0$ such that $\|T^{\ast}x\|\geq m\|x\|$ for all $x\in\mathcal{K}$.
		\item [(iii)] $T^{\ast}$ is bounded below with respect to the inner product, i.e., there is $m'>0$ such that $\langle T^{\ast}x,T^{\ast}x\rangle _{\mathcal{A}} \geq m'\langle x,x\rangle _{\mathcal{A}} $ for all $x\in\mathcal{K}$.
	\end{itemize}
\end{lemma}

\section{Controlled ${\ast}$-operator frame for $End_{\mathcal{A}}^{\ast}(\mathcal{H})$}
We begin this section with the following definition.
\begin{definition}\cite{haja}
	A family of adjointable operators $\{T_{i}\}_{i\in I}$ on a Hilbert $\mathcal{A}$-module $\mathcal{H}$ over a unital $C^{\ast}$-algebra is said to be an operator frame for $End_{\mathcal{A}}^{\ast}(\mathcal{H})$, if there exist two positives constants $A, B > 0$ such that 
\begin{equation}\label{eq3}
		A\langle x,x\rangle _{\mathcal{A}} \leq\sum_{i\in I}\langle T_{i}x,T_{i}x\rangle\leq B\langle x,x\rangle_{\mathcal{A}} ,  x\in\mathcal{H}.
	\end{equation}
	The numbers $A$ and $B$ are called lower and upper bound of the operator frame, respectively. If $A=B=\lambda$, the operator frame is $\lambda$-tight.\\
	 If $A = B = 1$, it is called a normalized tight operator frame or a Parseval operator frame.\\
	 If only upper inequality of \eqref{eq3} hold, then $\{T_{i}\}_{i\in I}$ is called an operator Bessel sequence for $End_{\mathcal{A}}^{\ast}(\mathcal{H})$.\\	
	If the sum in the middle of \eqref{eq3} is convergent in norm, the operator frame is called standard.
\end{definition}
Throughout the paper, series like \eqref{eq3} are assumed to be convergent in the norm sense.\\
Let $GL^{+}(\mathcal{H})$ be the set for all positive bounded linear invertible operators on $\mathcal{H}$ with bounded inverse.

\begin{definition}\cite{meknes}
	Let $C,C^{'} \in GL^{+}(H)$, a family of adjointable operators $\{T_{i}\}_{i\in I}$ on a Hilbert $\mathcal{A}$-module $\mathcal{H}$ over a unital $C^{\ast}$-algebra is said to be a $(C,C^{'})$-controlled operator frame for $End_{\mathcal{A}}^{\ast}(\mathcal{H})$, if there exist two positives constants $A,B > 0$ such that 
	\begin{equation}\label{1}
		A\langle x,x\rangle _{\mathcal{A}} \leq\sum_{i\in I}\langle T_{i}Cx,T_{i}C^{'}x\rangle_{\mathcal{A}} \leq B\langle x,x\rangle_{\mathcal{A}} , \,\,\,\, x\in\mathcal{H}.
	\end{equation}
	The elements $A$ and $B$ are called lower and upper bounds of the $(C,C^{'})$-controlled operator frame , respectively.\\
	 If $A=B=\lambda$, the $(C,C^{'})$-controlled operator frame  is $\lambda$-tight.\\
	 If $A = B = 1$, it is called a normalized tight $(C,C^{'})$-controlled operator frame or a Parseval $(C,C^{'})$-controlled operator frame .\\
	 If only upper inequality of \eqref{1} hold, then $\{T_{i}\}_{i\in i}$ is called a $(C,C^{'})$-controlled operator Bessel sequence for $End_{\mathcal{A}}^{\ast}(\mathcal{H})$.
\end{definition}
\begin{definition}\cite{haja}
	A family of adjointable operators $\{T_{i}\}_{i\in I}$ on a Hilbert $\mathcal{A}$-module $\mathcal{H}$ over a unital $C^{\ast}$-algebra is said to be a ${\ast}$ - operator frame for $End_{\mathcal{A}}^{\ast}(\mathcal{H})$, if there exist two positives constants $A$ and $B$ in $\mathcal{A}$ such that 
	\begin{equation}\label{kamos}
	A\langle x,x\rangle_{\mathcal{A}} \leq\sum_{i\in I}\langle T_{i}x,T_{i}x\rangle_{\mathcal{A}} \leq B\langle x,x\rangle_{\mathcal{A}} , , \,\,\,\, x\in\mathcal{H}.
	\end{equation}
	The elements $A$ and $B$ are called lower and upper bounds of the ${\ast}$-operator frame, respectively. If $A=B=\lambda$, the ${\ast}$-operator frame is $\lambda$-tight. If $A = B = 1_{\mathcal{A}}$, it is called a normalized tight ${\ast}$-operator frame or a Parseval ${\ast}$-operator frame. If only upper inequality of (\ref{kamos}) hold, then $\{T_{i}\}_{i\in I}$is called an ${\ast}$-operator Bessel sequence for $End_{\mathcal{A}}^{\ast}(\mathcal{H})$.
	
\end{definition}
\begin{definition}
	Let $C,C^{'} \in GL^{+}(H)$, a family of adjointable operators $\{T_{i}\}_{i\in I}$ on a Hilbert $\mathcal{A}$-module $\mathcal{H}$ over a unital $C^{\ast}$-algebra is said to be an $(C,C^{'})$-controlled $\ast$- operator frame for $End_{\mathcal{A}}^{\ast}(\mathcal{H})$, if there exist two positives constants $A$ and $B$ in $\mathcal{A}$ such that 
	\begin{equation}\label{1}
	A\langle x,x\rangle _{\mathcal{A}} A^{\ast} \leq\sum_{i\in I}\langle T_{i}Cx,T_{i}C^{'}x\rangle\leq B\langle x,x\rangle_{\mathcal{A}} B^{\ast},  x\in\mathcal{H}.
	\end{equation}
	The elements $A$ and $B$ are called lower and upper bounds of the $(C,C^{'})$-controlled $\ast$-operator frame , respectively.\\
	If $A=B=\lambda$, the $(C,C^{'})$-controlled operator frame  is $\lambda$-tight.\\
	If $A = B = 1_{\mathcal{A}}$, it is called a normalized tight $(C,C^{'})$-controlled $\ast$-operator frame or a Parseval $(C,C^{'})$-controlled $\ast$- operator frame .\\
	If only upper inequality of \eqref{1} hold, then $\{T_{i}\}_{i\in i}$ is called an $(C,C^{'})$-controlled $\ast$-operator Bessel sequence for $End_{\mathcal{A}}^{\ast}(\mathcal{H})$.
\end{definition}
\begin{example}
	Let $\mathcal{A}=l^{\infty}$ be the unitary $C^{\ast}$-algebra of all bounded complex-valued sequences and let $\mathcal{H}=C_0$ the set of all sequences converging to zero equipped with the $\mathcal{A}=l^{\infty}$-inner product:
	\begin{equation*}
	\langle x,y\rangle_{\mathcal{A}} =\langle (x_{i})_{i\in \mathbb{N}} , (y_{i})_{i\in \mathbb{N}}\rangle_{\mathcal{A}}=(x_{i}\bar{y_{i}})_{i\in \mathbb{N}},\quad for \; all \quad x=(x_{i})_{i\in \mathbb{N}}, y=(y_{i})_{i\in \mathbb{N}} \in \mathcal{H}
	\end{equation*}
	It's clear to see that $\mathcal{H}$ is a Hilbert $C^{\ast}$-module over $\mathcal{A}=l^{\infty}$.
	
	Let $j\in \mathbb{N}$ and $(a_{i})_{i\in \mathbb{N}}=(1+\frac{1}{i})_{i\in \mathbb{N}}$, we define $T_{j} \in End_{\mathcal{A}}^{\ast}(\mathcal{H}) $ by:
	\begin{equation*}
T_{j}((x_{i})_{i\in \mathbb{N}})=(\delta_{ij}a_{j}x_{j})_{i\in \mathbb{N}}\qquad (x_{i})_{i\in \mathbb{N}}\in \mathcal{H}
	\end{equation*}
	Let $\alpha \in \mathbb{R}^{\ast}$, we define two operators $C$ and $C^{'}$ on $\mathcal{H}$ by:
	\begin{align*}
	C:\mathcal{H} &\longrightarrow \mathcal{H}\\
	x&\longrightarrow Cx=\alpha x
	\end{align*}
	
	\begin{align*}
	C{'}:\mathcal{H} &\longrightarrow \mathcal{H}\\
	x&\longrightarrow C^{'}x=\frac{1}{\alpha} x
	\end{align*}
	We have,
\begin{align*}
    \sum_{j\in \mathbb{N}}\langle T_{j}Cx,T_{j}C^{'}x\rangle&=\langle (\delta_{ij}a_j  \alpha x_j)_{i\in \mathbb{N}}, (\delta_{ij}a_j \frac{1}{\alpha}x_j)_{i\in \mathbb{N}}\rangle_{\mathcal{A}}\\
    &= ((1+\frac{1}{i} )^2 x_i \bar{x_i})_{i\in \mathbb{N}}\\
    &=(1+\frac{1}{i} )_{i\in \mathbb{N}} \langle x, x \rangle_{\mathcal{A}} (1+\frac{1}{i} )_{i\in \mathbb{N}}.
\end{align*}
    Therefore $(T_{j})_{j\in \mathbb{N}}$ is a $(C,C^{'})$-controlled tight $\ast$-operator frame for $End_{\mathcal{A}}^{\ast}(\mathcal{H})$.

\end{example}
\begin{proposition}
	Every $(C,C^{'})$-controlled operator frame for $End_{\mathcal{A}}^{\ast}(\mathcal{H})$ is a  $(C,C^{'})$-controlled $\ast$- operator frame.
	
\end{proposition}
\begin{proof}
	Let $\{T_{i}\}_{i\in I}$ be a $(C,C^{'})$-controlled $\ast$- operator frame for $End_{\mathcal{A}}^{\ast}(\mathcal{H})$.\\
	Then, there exist two positives constants $A,B > 0$ such that 
	\begin{equation}
	A\langle x,x\rangle _{\mathcal{A}} \leq\sum_{i\in I}\langle T_{i}Cx,T_{i}C^{'}x\rangle_{\mathcal{A}} \leq B\langle x,x\rangle_{\mathcal{A}} , \,\,\,\, x\in\mathcal{H}.
	\end{equation}
	Hence 
	\begin{equation}
	(\sqrt{A}) 1_{\mathcal{A}}\langle x,x\rangle _{\mathcal{A}}((\sqrt{A}) 1_{\mathcal{A}})^{\ast} \leq\sum_{i\in I}\langle T_{i}Cx,T_{i}C^{'}x\rangle_{\mathcal{A}} \leq (\sqrt{B}) 1_{\mathcal{A}}\langle x,x\rangle_{\mathcal{A}}((\sqrt{B}) 1_{\mathcal{A}})^{\ast} , \,\,\,\, x\in\mathcal{H}.
	\end{equation}
	Therfore $\{T_{i}\}_{i\in I}$ is a $(C,C^{'})$-controlled $\ast$- operator frame for $End_{\mathcal{A}}^{\ast}(\mathcal{H})$ with bounds $(\sqrt{A}) 1_{\mathcal{A}}$ and $(\sqrt{B}) 1_{\mathcal{A}}$.
	
\end{proof}
Let $\{T_{i}\}_{i\in I}$  be a $(C,C^{'})$-controlled operator frame for $ End_{\mathcal{A}}^{\ast}(\mathcal{H})$.\\

The bounded linear operator  $T_{CC^{'}}:l^{2}(\mathcal{H}) \longrightarrow \mathcal{H}$ given by
\begin{equation*}
T_{(C,C^{'})}(\{y_{i}\}_{i\in I})=\sum_{i\in I}(CC^{'})^{\frac {1}{2}}T^{\ast}_{i}y_{i} , \{y_{i}\}_{i\in I} \in l^{2}(\mathcal{H})
\end{equation*}
is called the synthesis operator for the  $(C,C^{'})$-controlled $\ast$-operator frame $\{T_{i}\}_{i\in I}$.\\
The adjoint operator $T^{\ast}_{(C,C^{'})}: \mathcal{H}\rightarrow l^{2}(\{\mathcal{H}\})$ given by
\begin{equation}\label{2}
T^{\ast}_{(C,C^{'})}(x)=\{T_{i}(C^{'}C)^{\frac{1}{2}}x\}_{i \in  I} , x\in \mathcal{H}
\end{equation}
is called the analysis operator for the  $(C,C^{'})$-controlled $\ast$-operator frame $\{T_{i}\}_{i\in I}$ .\\
When $C$ and $C^{'}$ commute with each other, and commute with the operator $T^{\ast}_{i}T_{i}$ for each $i\in I$, then the $(C,C^{'})$-controlled frames operator: \\

$S_{(C,C^{'})}:\mathcal{H}\longrightarrow \mathcal{H}$ is defined as: 
$S_{(C,C^{'})}x=T_{(C,C^{'})}T^{\ast}_{(C,C^{'})}x=\sum_{i\in I}C^{'}T^{\ast}_{i}T_{i}Cx $\\
From now on we assume that $C$ and $C^{'}$ commute with each other, and commute with the operator $T^{\ast}_{i}T_{i}$ for each $i \in I$.
\begin{proposition}\label{P1}
	The $(C,C^{'})$-controlled frame operator $
	S_{(C,C^{'})}$ is bounded, positive, sefladjoint and invertible.
\end{proposition}
\begin{proof}
	As $\{T_{i}\}_{i\in I}$  is a $(C,C^{'})$-controlled $\ast$-operator frame,then\\
	$$\sum_{i\in I}\langle T_{i}Cx,T_{i}C^{'}x\rangle_{\mathcal{A}}=\langle \sum_{i\in I}C^{'}T^{\ast}_{i}T_{i}Cx,x\rangle_{\mathcal{A}}=\langle S_{CC^{'}}x,x\rangle_{\mathcal{A}}. $$
	It is clear that $S_{CC^{'}}$ is positive, bounded and linear operator. \\
	We have
\begin{align*}
    \langle S_{(C,C^{'})}x,y\rangle_{\mathcal{A}}&=\langle \sum_{i\in I}C^{'}T^{\ast}_{i}T_{i}Cx,y\rangle_{\mathcal{A}}\\
    &= \sum_{i\in I}\langle C^{'}T^{\ast}_{i}T_{i}Cx,y\rangle_{\mathcal{A}}\\
    &= \sum_{i\in I}\langle x,CT^{\ast}_{i}T_{i}C^{'}y\rangle_{\mathcal{A}}\\
    &=\langle x,\sum_{i\in I}CT^{\ast}_{i}T_{i}C^{'}y\rangle_{\mathcal{A}}\\
    &=\langle x,S_{(C^{'},C)}y\rangle_{\mathcal{A}}.
\end{align*}
	Therefore $S_{(C,C^{'})}^{\ast}=S_{(C^{'},C)}$. Since C and $C^{'}$ commute with each other and commute with $T^{\ast}_{i}T_{i}$ we have $S_{(C,C^{'})}$ selfadjoint. From the definition of controlled $\ast$-operator frame we have 
	$$A \langle x,x \rangle_{\mathcal{A}} A^{\ast} \leq \langle S_{(C,C^{'})}x,x\rangle_{\mathcal{A}} \leq B \langle x,x \rangle_{\mathcal{A}} B^{\ast}. $$
	So 
	$$A. Id_{\mathcal{H}}.A^{\ast}\leq S_{(C,C^{'})} \leq B .Id_{\mathcal{H}}.B^{\ast} .$$
	Where $Id_{\mathcal{H}}$ is the identity operator in ${\mathcal{H}}$. thus $S_{(C,C^{'})}$ is invertible.
\end{proof}

\begin{theorem}\label{do51}
	Let $\{T_{i}\}_{i\in I} \in End_{\mathcal{A}}^{\ast}(\mathcal{H})$, and $\sum_{i\in I}\langle T_{i}Cx,T_{i}C^{'}x\rangle_{\mathcal{A}}$ converge in norm ${\mathcal{A}}$. Then $\{T_{i}\}_{i\in I}$ is a $(C,C^{'})$-controlled $\ast$-operator frame if and only if
\begin{equation}\label{do1}
    \|A^{-1}\|^{-2} \|\langle x,x \rangle_{\mathcal{A}}\| \leq \|\sum_{i\in I}\langle T_{i}Cx,T_{i}C^{'}x\rangle_{\mathcal{A}}\| \leq \|B\|^2 \|\langle x,x \rangle_{\mathcal{A}}\|.
\end{equation}
	for every $x \in {\mathcal{H}}$ and stricly nonzero elements $A,B \in {\mathcal{A}} $.
\end{theorem}
\begin{proof}
	Suppose that $\{T_{i}\}_{i\in I}$ is a $(C,C^{'})$-controlled $\ast$-operator frame. Then $$\langle x,x \rangle_{\mathcal{A}} \leq A^{-1} \langle S_{(C,C^{'})}x,x \rangle_{\mathcal{A}} (A^{\ast})^{-1},$$
	and $$\langle S_{(C,C^{'})}x,x \rangle_{\mathcal{A}} \leq B \langle x,x \rangle_{\mathcal{A}} B^{\ast}.$$
	Hence $$\|A^{-1}\|^{-2} \|\langle x,x \rangle_{\mathcal{A}}\| \leq \|\sum_{i\in I}\langle T_{i}Cx,T_{i}C^{'}x\rangle_{\mathcal{A}}\| \leq \|B\|^2 \|\langle x,x \rangle_{\mathcal{A}}\|.$$
	Converselly, assume that (\ref{do1}) holds. From (\ref{P1}), the  $(C,C^{'})$-controlled frame operator $S_{(C,C^{'})}$ is positive, selfadjoint and invertible. Hence 
\begin{equation}\label{do2}
    \langle (S_{(C,C^{'})})^{\frac{1}{2}}x,(S_{(C,C^{'})})^{\frac{1}{2}}x \rangle_{\mathcal{A}}=  \langle S_{(C,C^{'})}x,x\rangle_{\mathcal{A}}=\sum_{i\in I}\langle T_{i}Cx,T_{i}C^{'}x\rangle_{\mathcal{A}}.
\end{equation}
    Using (\ref{do1}) and (\ref{do2}), we get
\begin{equation}\label{do3}
    \|A^{-1}\|.\|x\|\leq \|(S_{(C,C^{'})})^{\frac{1}{2}}x\|\leq \|B\|. \|x\|.
\end{equation}
    Using (\ref{do3}) and Lemma (\ref{haja4}), we conclude that $\{T_{i}\}_{i\in I}$ is a $(C,C^{'})$-controlled $\ast$-operator frame $End_{\mathcal{A}}^{\ast}(\mathcal{H})$.
	 
\end{proof}
    The following theorem shows that any $\ast$-operator frame is a $C^2$-controlled $\ast$-operator frame for ${\mathcal{H}}$ and vice versa.
\begin{theorem}
	Let $C \in GL^{+}(\mathcal{H})$. The family $\{T_{i}\}_{i\in I} \in End_{\mathcal{A}}^{\ast}(\mathcal{H})$ is a $\ast$-operator frame for $End_{\mathcal{A}}^{\ast}(\mathcal{H})$ if and only if 
	$\{T_{i}\}_{i\in I}$ is a $C^2$-controlled $\ast$-operator frame.
\end{theorem}
\begin{proof}
	Let $\{T_{i}\}_{i\in I}$ be a $C^2$-controlled $\ast$-operator frame with bounds A and B. Then $$\sum_{i\in I}\langle T_{i}Cx,T_{i}Cx\rangle_{\mathcal{A}}\leq B \langle x,x \rangle_{\mathcal{A}} B^{\ast},\,\, x \in {\mathcal{H}}. $$
	on one hand, for every $x \in {\mathcal{H}}$ we have,
\begin{align*}
    A \langle x,x \rangle_{\mathcal{A}} A^{\ast}&=A \langle CC^{-1}x,CC^{-1}x \rangle_{\mathcal{A}} A^{\ast}\\
    &\leq A \|C\|^2\langle C^{-1}x,C^{-1}x \rangle_{\mathcal{A}} A^{\ast}\\
    &\leq \|C\|^2 \sum_{i\in I}\langle T_{i}CC^{-1}x,T_{i}CC^{-1}x\rangle_{\mathcal{A}}\\
    &= \|C\|^2 \sum_{i\in I}\langle T_{i}x,T_{i}x\rangle_{\mathcal{A}}.
\end{align*}
    Then $$A \|C\|^{-1}\langle x,x \rangle_{\mathcal{A}} A^{\ast}\|C\|^{-1}\leq \sum_{i\in I}\langle T_{i}x,T_{i}x\rangle_{\mathcal{A}}.$$
	On the other hand for any $x \in {\mathcal{H}}$ we have 
\begin{align*}
   \sum_{i\in I}\langle T_{i}x,T_{i}x\rangle_{\mathcal{A}} &=\sum_{i\in I}\langle T_{i}CC^{-1}x,T_{i}CC^{-1}x\rangle_{\mathcal{A}}\\
    &\leq  B \langle C^{-1}x,C^{-1}x \rangle_{\mathcal{A}} B^{\ast}\\
    &\leq B \|C^{-1}\|^2\langle x,x \rangle_{\mathcal{A}}B^{\ast}.
\end{align*}
	Then $$A \|C\|^{-1}\langle x,x \rangle_{\mathcal{A}} A^{\ast}\|C\|^{-1}\leq \sum_{i\in I}\langle T_{i}x,T_{i}x\rangle_{\mathcal{A}} \leq B \|C\|^{-1}\langle x,x \rangle_{\mathcal{A}} B^{\ast}\|C\|^{-1} .$$
	Therefore $\{T_{i}\}_{i\in I}$ is a $\ast$-operator frame with bounds $A \|C\|^{-1}$  and $B \|C\|^{-1}$.\\
	For the converse, suppose that $\{T_{i}\}_{i\in I}$ is a $\ast$-operator frame with bounds M and N. on one hand we have for any $x \in {\mathcal{H}}$,
	$$M\langle x,x \rangle_{\mathcal{A}} M^{\ast}\leq \sum_{i\in I}\langle T_{i}x,T_{i}x\rangle_{\mathcal{A}} \leq N \langle x,x \rangle_{\mathcal{A}} N^{\ast} .$$
	Thus, for all $x \in {\mathcal{H}}$,
\begin{align*}
    \sum_{i\in I}\langle T_{i}Cx,T_{i}Cx\rangle_{\mathcal{A}}&\leq N \|C\|^2\langle x,x \rangle_{\mathcal{A}} N^{\ast}\\
    &=N \|C\|\langle x,x \rangle_{\mathcal{A}} N^{\ast}\|C\|.
\end{align*}
    On the other hand, we have 
\begin{align*}
    M\langle x,x \rangle_{\mathcal{A}} M^{\ast}&=M\langle C^{-1}Cx,C^{-1}Cx \rangle_{\mathcal{A}} M^{\ast}\\
    &\leq \|C^{-1}\|^2 \sum_{i\in I}\langle T_{i}Cx,T_{i}Cx\rangle_{\mathcal{A}}.
\end{align*}
    Therefore 
    $$M\|C^{-1}\|^{-1}\langle x,x \rangle_{\mathcal{A}} M^{\ast}\|C^{-1}\|^{-1}\leq \sum_{i\in I}\langle T_{i}Cx,T_{i}Cx\rangle_{\mathcal{A}}\leq N \|C\|\langle x,x \rangle_{\mathcal{A}} N^{\ast}\|C\|. $$
	This gives that $\{T_{i}\}_{i\in I}$ is a $C^2$-controlled $\ast$-operator frame with bounds $M\|C^{-1}\|^{-1}$ and $N^{\ast}\|C\|$.
\end{proof}
\begin{proposition}
	Let $\{T_{i}\}_{i\in I}$ be an $\ast$-operator frame for $End_{\mathcal{A}}^{\ast}(\mathcal{H})$ with frmar operator S and $C,C^{'} \in GL^{+}(H)$. Then $\{T_{i}\}_{i\in I}$ is a $(C,C^{'})$-controlled $\ast$-operator frame for $End_{\mathcal{A}}^{\ast}(\mathcal{H})$.

\end{proposition}
\begin{proof}
	Let $\{T_{i}\}_{i\in I}$ be an $\ast$-operator frame with bounds A and B. Then by (\ref{eq3}) we have 
	$$\langle x,x \rangle_{\mathcal{A}} \leq A^{-1}\langle Sx,x \rangle_{\mathcal{A}}(A^{\ast})^{-1},\,\,\,\,\,\,\,\langle Sx,x \rangle_{\mathcal{A}}\leq B\langle x,x \rangle_{\mathcal{A}}B^{\ast} .$$
	Hence 
\begin{equation}\label{do5}
    \|A^{-1}\|^{-2} \|\langle x,x \rangle_{\mathcal{A}}\| \leq \|\sum_{i\in I}\langle T_{i}x,T_{i}x\rangle_{\mathcal{A}}\| \leq \|B\|^2 \|\langle x,x \rangle_{\mathcal{A}}\|.
\end{equation}
	
	We have 
	$$\|\sum_{i\in I}\langle T_{i}Cx,T_{i}C^{'}x\rangle_{\mathcal{A}}\|=\|\langle S_{(C,C^{'})}x,x \rangle_{\mathcal{A}}\|,$$
	and 
\begin{equation}\label{do4}
    \|\sum_{i\in I}\langle T_{i}Cx,T_{i}C^{'}x\rangle_{\mathcal{A}}\|=\|C\|.\|C^{'}\|.\|\langle Sx,x \rangle_{\mathcal{A}}\|.
\end{equation}
    Using (\ref{do5}) and (\ref{do4}), we have

\begin{align*}
    \|A^{-1}\|^{-2} \|C\|.\|C^{'}\| \|\langle x,x \rangle_{\mathcal{A}}\|\leq \|\sum_{i\in I}\langle T_{i}Cx,T_{i}C^{'}x\rangle_{\mathcal{A}}\|\leq \|B\|^2 \|C\|.\|C^{'}\| \|\langle x,x \rangle_{\mathcal{A}}\|.
\end{align*}

    Therefore, from theorem (\ref{do51}), we have $\{T_{i}\}_{i\in I}$ is a $(C,C^{'})$-controlled $\ast$-operator frame with bounds $\|A^{-1}\|^{-1} \|C\|^{\frac{1}{2}}.\|C^{'}\|^{\frac{1}{2}}$  and $\|B\| \|C\|^{\frac{1}{2}}.\|C^{'}\|^{\frac{1}{2}} $.
\end{proof}
\begin{theorem}
	Let $C,C^{'} \in GL^{+}(H)$, $\{T_{i}\}_{i\in I} \in End_{\mathcal{A}}^{\ast}(\mathcal{H})$. Suppose that $C,C^{'}$ commute with each other and commute with $T_{i}T_{i}^{\ast}$ for all $i\in I$. The family $\{T_{i}\}_{i\in I}$ is a $(C,C^{'})$-controlled $\ast$-operator Bessel sequence for $End_{\mathcal{A}}^{\ast}(\mathcal{H})$ with bound B if and only if  the operator $T_{CC^{'}}:l^{2}(\mathcal{H}) \longrightarrow \mathcal{H}$ given by
	\begin{equation*}
	T_{CC^{'}}(\{y_{i}\}_{i\in I})=\sum_{i\in I}(CC^{'})^{\frac {1}{2}}T^{\ast}_{i}y_{i} \qquad  \forall \{y_{i}\}_{i\in I} \in l^{2}(\mathcal{H}).
	\end{equation*}
	is well defined and bounded operator with $\|T_{CC^{'}}\|\leq \|B\|$. 
	
\end{theorem}
\begin{proof}
	Assume that $\{T_{i}\}_{i\in I}$ is a $(C,C^{'})$-controlled $\ast$-operator Bessel sequence for $End_{\mathcal{A}}^{\ast}(\mathcal{H})$ with bound B. As a result of (\ref{do5}) 
	$$ \|\sum_{i\in I}\langle T_{i}Cx,T_{i}C^{'}x\rangle_{\mathcal{A}}\| \leq \|B\|^2 \|\langle x,x \rangle_{\mathcal{A}}\|.$$
	We have
	\begin{align*}
	\|T_{(C,C^{'})}(\{y_{i}\}_{i\in I})\|^{2}&=\underset{x \in \mathcal{H}, \|x\|=1}{\sup}\|\langle \sum_{i\in I}(CC^{'})^{\frac {1}{2}} T^{\ast}_{i}y_{i},x\rangle_{\mathcal{A}}\|^{2}\\
	&=\underset{x \in \mathcal{H}, \|x\|=1}{\sup}\| \sum_{i\in I}\langle(CC^{'})^{\frac {1}{2}}T^{\ast}_{i}y_{i},x\rangle_{\mathcal{A}} \|^{2}\\ 
	&=\underset{x \in U, \|x\|=1}{\sup}\| \sum_{i\in I}\langle y_{i},T_{i}(CC^{'})^{\frac {1}{2}}x\rangle_{\mathcal{A}} \|^{2}\\ 
	&\leq \underset{x \in U, \|x\|=1}{\sup}\| \sum_{i\in I}\langle y_{i},y_{i}\rangle_{\mathcal{A}} \| \|\sum_{i\in I}\langle T_{i}(CC^{'})^{\frac {1}{2}}x,T_{i}(CC^{'})^{\frac {1}{2}}x\rangle_{\mathcal{A}} \|\\ 
	&= \underset{x \in U, \|x\|=1}{\sup}\| \sum_{i\in I}\langle y_{i},y_{i}\rangle_{\mathcal{A}} \| \|\sum_{i\in I}\langle T_{i}Cx,T_{i}C^{'}x\rangle_{\mathcal{A}} \|\\ 
	&\leq \underset{x \in \mathcal{H}, \|x\|=1}{\sup}\| \sum_{i\in I}\langle y_{i},y_{i}\rangle_{\mathcal{A}} \|\|B\|^{2}\|x\|^{2} =\|B\|^{2}\|\{y_{i}\}_{i \in I}\|^{2}.
	\end{align*}
	Then, the sum $\sum_{i\in I}\langle(CC^{'})^{\frac {1}{2}}T^{\ast}_{i}y_i$ is convergent and we have $$\|T_{(C,C^{'})}(\{y_{i}\}_{i\in I})\|^{2}\leq \|B\|^{2}\|\{y_{i}\}_{i \in I}\|^{2}.$$
	Hence $$\|T_{(C,C^{'})}\|^{2}\leq \|B\|^{2}.$$
	Thus the operator $T_{(C,C^{'})}$ is well defined, bounded and $$\|T_{(C,C^{'})}\|\leq \|B\|.$$
	For the converse, suppose that the operator $T_{(C,C^{'})}$ is well defined, bounded and $\|T_{(C,C^{'})}\|\leq \|B\|.$
	For all $x \in {\mathcal{H}} $, we have 
\begin{align*}
    \sum_{i\in I}\langle T_{i}Cx,T_{i}C^{'}x\rangle_{\mathcal{A}}&=\sum_{i\in I}\langle C^{'} T_{i}^{\ast}T_{i}Cx,x\rangle_{\mathcal{A}}\\
    &=\sum_{i\in I}\langle (CC^{'})^{\frac{1}{2}} T_{i}^{\ast}T_{i}(CC^{'})^{\frac{1}{2}}x,x\rangle_{\mathcal{A}}\\
    &=\langle T_{(C,C^{'})}(\{y_{i}\}_{i\in I}),x\rangle_{\mathcal{A}}\\
    &\leq \|T_{(C,C^{'})}\| \|(\{y_{i}\}_{i\in I})\| \|x\|\\
    &\leq \|T_{(C,C^{'})}\| (\sum_{i\in I}\|T_{i}(CC^{'})^{\frac{1}{2}}x\|^2)^{\frac{1}{2}}  \|x\|\\
    &=\|T_{(C,C^{'})}\|(\sum_{i\in I}\langle T_{i}Cx,T_{i}C^{'}x\rangle_{\mathcal{A}})^{\frac{1}{2}} \|x\|
\end{align*}
    where $y_i=T_{i}(CC^{'})^{\frac{1}{2}}x $.\\
    Therefore $$\sum_{i\in I}\langle T_{i}Cx,T_{i}C^{'}x\rangle_{\mathcal{A}}\leq \|T_{(C,C^{'})}\|^2 \|x\|^2.$$
    Hence $$\sum_{i\in I}\langle T_{i}Cx,T_{i}C^{'}x\rangle_{\mathcal{A}}\leq \|B\|^2 \|x\|^2,\,\,\,\ as\ ,\,\,\,\,\,\|T_{(C,C^{'})}\|\leq \|B\|.$$
    This give that $\{T_{i}\}_{i\in I}$ is a $(C,C^{'})$-controlled $\ast$-operator Bessel sequence for $End_{\mathcal{A}}^{\ast}(\mathcal{H})$.
\end{proof}
\begin{theorem}\label{do8}
	Let $\{T_{i}\}_{i\in I} \in End_{\mathcal{A}}^{\ast}(\mathcal{H})$ be a $(C,C^{'})$-controlled $\ast$-operator frame for $End_{\mathcal{A}}^{\ast}(\mathcal{H})$ with bounds A and B, with operator frame $S_{(C,C^{'})}$. Let $\theta \in End_{\mathcal{A}}^{\ast}(\mathcal{H}) $ be injective and has a closed range. Suppose that $\theta$ commute with $C$ and $C^{'}$. Then $\{T_{i}\theta\}_{i\in I}$ is a $(C,C^{'})$-controlled $\ast$-operator frame for $End_{\mathcal{A}}^{\ast}(\mathcal{H})$ with operator frame $\theta^{\ast}S_{(C,C^{'})}\theta$ with bounds $\|(\theta^{\ast}\theta)^{-1}\|^{\frac{-1}{2}}A$ and $\|\theta\|B$.
	
\end{theorem}
\begin{proof}
	Let $\{T_{i}\}_{i\in I} \in End_{\mathcal{A}}^{\ast}(\mathcal{H})$ be a $(C,C^{'})$-controlled $\ast$-operator frame for $End_{\mathcal{A}}^{\ast}(\mathcal{H})$ with bounds A and B.
	Then
\begin{equation}\label{do6}
    A \langle \theta x,\theta x \rangle_{\mathcal{A}} A^{\ast}\leq \sum_{i\in I}\langle T_{i}C\theta x,T_{i}C^{'} \theta x\rangle_{\mathcal{A}}\leq B \langle \theta x,\theta x \rangle_{\mathcal{A}} B^{\ast}.
\end{equation}

	From lemma (\ref{haja2}), we have 
	$$\|(\theta^{\ast}\theta)^{-1}\|^{-1} \langle \ x, x \rangle_{\mathcal{A}}\leq \langle \theta x,\theta x \rangle_{\mathcal{A}},\,\,\, x \in {\mathcal{H}}.$$
	Hence 
\begin{equation}\label{do7}
    \|(\theta^{\ast}\theta)^{-1}\|^{\frac{-1}{2}} A \langle  x, x \rangle_{\mathcal{A}} (\|(\theta^{\ast}\theta)^{-1}\|^{\frac{-1}{2}}A)^{\ast}\leq A \langle \theta x,\theta x \rangle_{\mathcal{A}} A^{\ast}.
\end{equation}
    Since 
    $$\langle \theta x,\theta x \rangle_{\mathcal{A}}\leq \|\theta\|^2 \langle \theta x,\theta x \rangle_{\mathcal{A}},$$
	we have 
\begin{equation}\label{do8}
    B \langle \theta x,\theta x \rangle_{\mathcal{A}} B^{\ast} \leq\|\theta\| B \langle  x, x \rangle_{\mathcal{A}} (\|\theta\| B)^{\ast},\,\,\, x \in {\mathcal{H}}
\end{equation}
	Using (\ref{do6}), (\ref{do7}), (\ref{do8}) we have 
	$$\|(\theta^{\ast}\theta)^{-1}\|^{\frac{-1}{2}} A \langle  x, x \rangle_{\mathcal{A}} (\|(\theta^{\ast}\theta)^{-1}\|^{\frac{-1}{2}}A)^{\ast} \leq \sum_{i\in I}\langle T_{i}C \theta x,T_{i}C^{'} \theta x\rangle_{\mathcal{A}} \leq \|\theta\| B \langle  x, x \rangle_{\mathcal{A}} (\|\theta\| B)^{\ast},\,\,\, x \in {\mathcal{H}}. $$
	Therefore $\{T_{i}\theta\}_{i\in I}$ is a $(C,C^{'})$-controlled $\ast$-operator frame for $End_{\mathcal{A}}^{\ast}(\mathcal{H})$.\\
	Moreover for every $x \in {\mathcal{H}}$, we have 
	$$\theta^{\ast}S_{(C,C^{'})}\theta = \theta^{\ast}\sum_{i\in I}C^{'}T^{\ast}_{i}T_{i}C\theta x=\sum_{i\in I}\theta^{\ast}C^{'}T^{\ast}_{i}T_{i}C\theta x=\sum_{i\in I}C^{'}(T_{i} \theta)^{\ast}(T_{i}\theta )C x.$$
	This completes the proof.
\end{proof}
\begin{corollary}
	 Let $\{T_{i}\}_{i\in I} \in End_{\mathcal{A}}^{\ast}(\mathcal{H})$ be a $(C,C^{'})$-controlled $\ast$-operator frame for $End_{\mathcal{A}}^{\ast}(\mathcal{H})$, with operator frame $ S_{(C,C^{'})}$. Then  $\{T_{i}S_{(C,C^{'})}^{-1}\}_{i\in I}$ is a $(C,C^{'})$-controlled $\ast$-operator frame for $End_{\mathcal{A}}^{\ast}(\mathcal{H})$.
\end{corollary}
\begin{proof}
	The proof is a result of (\ref{do8}) for $\theta = S^{-1}$.

\end{proof}
\begin{theorem}
	Let $\{T_{i}\}_{i\in I} \in End_{\mathcal{A}}^{\ast}(\mathcal{H})$ be a $(C,C^{'})$-controlled $\ast$-operator frame for $End_{\mathcal{A}}^{\ast}(\mathcal{H})$ with bounds A and B. Let $\theta \in End_{\mathcal{A}}^{\ast}(\mathcal{H}) $ be surjective. Then $\{\theta T_{i}\}_{i\in I}$ is a $(C,C^{'})$-controlled $\ast$-operator frame for $End_{\mathcal{A}}^{\ast}(\mathcal{H})$   with bounds $A\|(\theta \theta^{\ast})^{-1}\|^{\frac{-1}{2}}$ , $B\|\theta\|$.
	
\end{theorem}
\begin{proof}
	From the definition of $(C,C^{'})$-controlled $\ast$-operator frame, we have 
\begin{equation}\label{do9}
    A\langle x,x\rangle _{\mathcal{A}} A^{\ast} \leq\sum_{i\in I}\langle T_{i}Cx,T_{i}C^{'}x\rangle_{\mathcal{A}}\leq B\langle x,x\rangle_{\mathcal{A}} B^{\ast},  x\in\mathcal{H}.
\end{equation}
    Using Lemma(\ref{haja2}), we have 
\begin{equation}\label{douae}
    \|(\theta \theta^{\ast})^{-1}\|^{-1} \langle T_{i}Cx,T_{i}C^{'}x\rangle_{\mathcal{A}} \leq \langle \theta T_{i}Cx,\theta T_{i}C^{'}x \rangle_{\mathcal{A}} \leq \|\theta\|^2 \langle T_{i}Cx,T_{i}C^{'}x\rangle_{\mathcal{A}}
\end{equation}  
	From (\ref{do9}) and (\ref{douae}), we have
	$$A\|(\theta \theta^{\ast})^{-1}\|^{\frac{-1}{2}}\langle  x, x \rangle_{\mathcal{A}}(A\|(\theta \theta^{\ast})^{-1}\|^{\frac{-1}{2}})^{\ast}\leq \sum_{i\in I}\langle \theta T_{i}Cx,\theta T_{i}C^{'}x\rangle \leq B \|\theta\|\langle x,x\rangle_{\mathcal{A}} (B\|\theta\| )^{\ast},  x\in\mathcal{H}.  $$
	Hence $\{\theta T_{i}\}_{i\in I}$ is a $(C,C^{'})$-controlled $\ast$-operator frame for $End_{\mathcal{A}}^{\ast}(\mathcal{H})$.
\end{proof}
    Under wich conditions a controlled $\ast$- operator frame for $End_{\mathcal{A}}^{\ast}(\mathcal{H})$ with $\mathcal{H}$ a $C^{\ast}$-module over a unital $C^{\ast}$-algebras $\mathcal{A}$ is also a controlled $\ast$- operator frame for $End_{\mathcal{A}}^{\ast}(\mathcal{H})$ with $\mathcal{H}$ a $C^{\ast}$-module over a unital $C^{\ast}$-algebras $\mathcal{B}$ .
    The following theorem answer this questions.
\begin{theorem}
	Let $(\mathcal{H}, \mathcal{A}, \langle .,.\rangle_{\mathcal{A}})$ and $(\mathcal{H}, \mathcal{B}, \langle .,.\rangle_{\mathcal{B}})$ be two hilbert $C^{\ast}$-modules and let $\varphi$: $\mathcal{A} \longrightarrow \mathcal{B}$ be a $\ast$-homomorphisme and $\theta$ be a map on $\mathcal{H}$ such that $\langle \theta x,\theta y\rangle_{\mathcal{B}}=\varphi( \langle x,y\rangle_{\mathcal{A}})  $ for all $x, y \in \mathcal{H} $. Suppose $\{T_{i}\}_{i\in I} \subset End_{\mathcal{A}}^{\ast}(\mathcal{H})$ is a $(C,C^{'})$-controlled $\ast$-operator frame for $(\mathcal{H}, \mathcal{A}, \langle .,.\rangle_{\mathcal{A}})$ with frame operator $S_{\mathcal{A}}$ and lower and upper bounds A and B respectively. If $\theta$ is surjective such that  $\theta T_i=T_i \theta$ for each $i \in I$ and  $\theta C= C \theta$ and $\theta C^{'}= C^{'} \theta$, then $\{T_{i}\}_{i\in I}$ is a $(C,C^{'})$-controlled $\ast$-operator frame for $(\mathcal{H}, \mathcal{B}, \langle .,.\rangle_{\mathcal{B}})$ with frame operator $S_{\mathcal{B}}$  and lower and upper  bounds $\varphi (A)$,  $\varphi (B)$ respectively and  $\langle S_{\mathcal{B}}\theta x, \theta y\rangle_{\mathcal{B}}= \varphi ( \langle S_{\mathcal{A}} x,  y\rangle_{\mathcal{A}}) $.
	
\end{theorem}
\begin{proof}
	Since $\theta $ is surjective, then for every $y \in \mathcal{H}$ there exists $x \in \mathcal{H}$ such that $\theta x= y$. Using  the definition of $(C,C^{'})$-controlled $\ast$-operator frame we have,
	$$ A\langle x,x\rangle _{\mathcal{A}} A^{\ast} \leq\sum_{i\in I}\langle T_{i}Cx,T_{i}C^{'}x\rangle\leq B\langle x,x\rangle_{\mathcal{A}} B^{\ast},  x\in\mathcal{H}.$$
	By lemma (\ref{haja1}) we have 
	$$\varphi (A\langle x,x\rangle _{\mathcal{A}} A^{\ast} )\leq \varphi (\sum_{i\in I}\langle T_{i}Cx,T_{i}C^{'}x\rangle_{\mathcal{A}})\leq \varphi (B\langle x,x\rangle_{\mathcal{A}} B^{\ast}),  x\in\mathcal{H}.$$
	From the definition of $\ast$-homomorphisme we have 
	$$\varphi (A) \varphi (\langle x,x\rangle _{\mathcal{A}}) \varphi( A^{\ast} )\leq \varphi (\sum_{i\in I}\langle T_{i}Cx,T_{i}C^{'}x\rangle_{\mathcal{A}})\leq \varphi (B) \varphi (\langle x,x\rangle_{\mathcal{A}}) \varphi ( B^{\ast}),  x\in\mathcal{H}.$$
	Using the relation betwen $\theta$ and $\varphi $ we get 
	$$\varphi (A) \langle \theta x,\theta x\rangle _{\mathcal{B}} (\varphi( A))^{\ast} \leq \sum_{i\in I}\langle \theta T_{i}Cx,\theta T_{i}C^{'}x\rangle_{\mathcal{B}}\leq \varphi (B)  \langle \theta x, \theta x\rangle_{\mathcal{B}}) (\varphi ( B))^{\ast},  x\in\mathcal{H}.$$
	Since $\theta T_i=T_i \theta$  ,  $\theta C= C \theta$ and $\theta C^{'}= C^{'} \theta$ we have 
	$$\varphi (A) \langle \theta x,\theta x\rangle _{\mathcal{B}} (\varphi( A))^{\ast} \leq \sum_{i\in I}\langle  T_{i} C\theta x, T_{i} C^{'}\theta x\rangle_{\mathcal{B}}\leq \phi (B)  \langle \theta x, \theta x\rangle_{\mathcal{B}}) (\varphi ( B))^{\ast},  x\in\mathcal{H}.$$
	Therefore 
	$$\varphi (A) \langle y,y \rangle _{\mathcal{B}} (\varphi( A))^{\ast} \leq \sum_{i\in I}\langle  T_{i} C y, T_{i} C^{'}y\rangle_{\mathcal{B}}\leq \varphi (B)  \langle y, y\rangle_{\mathcal{B}}) (\varphi ( B))^{\ast},  y\in\mathcal{H}.$$
	This implies that $\{T_{i}\}_{i\in I}$ is a $(C,C^{'})$-controlled $\ast$-operator frame for $(\mathcal{H}, \mathcal{B}, \langle .,.\rangle_{\mathcal{B}})$ with bounds $\varphi (A)$ and $\varphi (B)$. 
	Moreover we have 
\begin{align*}
    \varphi ( \langle S_{\mathcal{A}} x,  y\rangle_{\mathcal{A}}&= \varphi ( \langle \sum_{i\in I}  T_{i}Cx, T_{i}C^{'}y \rangle_{\mathcal{A}})\\
    &=  \sum_{i\in I}\varphi ( \langle T_{i}Cx, T_{i}C^{'}y \rangle_{\mathcal{A}})\\
    &=\sum_{i\in I} \langle \theta T_{i}Cx, \theta T_{i}C^{'}y \rangle_{\mathcal{B}}\\
    &= \sum_{i\in I} \langle  T_{i}C\theta x,  T_{i}C^{'}\theta y \rangle_{\mathcal{B}}\\
    &= \langle \sum_{i\in I}C^{'} T_{i}^{\ast} T_{i}C\theta x,  \theta y \rangle_{\mathcal{B}}\\
    &=\langle S_{\mathcal{B}} \theta x,  \theta y\rangle_{\mathcal{A}}).
\end{align*}
	Which completes the proof.
\end{proof}

   ACKNOWLEDGMENTS\\
   Sincere thanks goes to the valuable comments of the referees and the editors that give a step forward to the main file of the manuscript. Also the authors would like to thanks the reviewers for their valuable comments.

\bibliographystyle{amsplain}

\end{document}